\documentclass[preprint,10pt]{elsarticle}

\usepackage{amssymb}
\usepackage{amsthm}
\usepackage{mathrsfs}
\usepackage{enumerate}
\usepackage{amsmath}
\usepackage{amsfonts}
\usepackage{amssymb}
\usepackage{graphicx}
\usepackage{titlesec}
\usepackage{etoolbox}
\makeatletter
\patchcmd{\ttlh@hang}{\parindent\z@}{\parindent\z@\leavevmode}{}{}
\patchcmd{\ttlh@hang}{\noindent}{}{}{}
\makeatother
\usepackage{titletoc}
\usepackage{fancyhdr}
\usepackage{indentfirst}
\usepackage[margin=0.8in,vmargin=2.7cm]{geometry}

\numberwithin{equation}{section}

\journal{ }

\newtheorem{theorem}{Theorem}[section]
\newtheorem{lemma}[theorem]{Lemma}

\begin{document}

\begin{frontmatter}



\title{A note on the Riemann solutions to the isentropic Euler equations in the vanishing pressure limit}

\author[Author1]{Sana Keita}    \ead{skeit085@uottawa.ca}
\author[Author1]{Yves Bourgault\corref{cor1}} \ead{ybourg@uottawa.ca}
\cortext[cor1]{Corresponding author.}
\address[Author1]{Department of Mathematics and Statistics, University of Ottawa, K1N 6N5, Ottawa, Ontario, Canada}

\begin{abstract}
The behaviour of the solutions to the Riemann problem for the isentropic Euler equations when the pressure vanishes is analysed. It is shown that any solution composed of a $1$-shock wave and a $2$-rarefaction wave tends to a two-shock wave when the pressure gets smaller than a fixed value determined by the Riemann data; by contrast, any solution composed of a $1$-rarefaction wave and  a $2$-shock wave tends to a two-rarefaction wave. The two situations are illustrated with numerical tests.
\end{abstract}


\end{frontmatter}

\section{Introduction}
The one-dimensional isentropic Euler equations of gas dynamics are
\begin{equation}
\left\lbrace
\begin{aligned}
&\partial_t\rho+\partial_x(\rho v)=0,\\
&\partial_t(\rho v)+\partial_x(\rho v^2+p)=0,
\end{aligned}
\right.
\label{EulerEqu}
\end{equation} 
where $\rho>0$, $v$ and $p$ are the density, velocity and pressure of the gas, respectively.  The  pressure depends on the density and is determined from the constitutive thermodynamic relations of the gas under consideration. We restrict ourselves to polytropic perfect gases for which the state equation for the pressure is given by
\begin{equation}
p=p(\rho)=\kappa\rho^{\gamma}, \quad \kappa>0,\text{ }\gamma>1.
\label{PressIsentrFlow}
\end{equation}
System \eqref{EulerEqu}-\eqref{PressIsentrFlow} describes the flow of isentropic compressible fluids. Formally, the limit system of \eqref{EulerEqu}  when the pressure vanishes is 
\begin{equation}
\left\lbrace
\begin{aligned}
&\partial_t\rho+\partial_x(\rho v)=0,\\
&\partial_t(\rho v)+\partial_x(\rho v^2)=0.
\end{aligned}
\right.
\label{NoPressEulerEqu}
\end{equation}
System \eqref{NoPressEulerEqu} is known as the \textit{zero-pressure gas dynamics system} \cite{Bouchut2} or as the \textit{sticky particle system} \cite{weinan1996} that arises in the modelling of particles hitting and sticking to each other to explain the formation of large scale structures in the universe. Under suitable generalized Rankine-Hugoniot relation and entropy condition, the Riemann problem for \eqref{NoPressEulerEqu} is constructively solved in \cite{Sheng}. The Riemann solution is either a $\delta$-shock or a two-contact-discontinuity with a vacuum state, depending on the Riemann data.



There are four possible configurations for a solution to the Riemann problem for \eqref{EulerEqu}-\eqref{PressIsentrFlow}: two-shock, two-rarefaction, $1$-shock combined with $2$-rarefaction or $1$-rarefaction combined with $2$-shock. For more details, see \cite{Smoller}.  
Chen and Liu  \cite{Vacuum} analysed the behaviour of Riemann solutions for \eqref{EulerEqu}-\eqref{PressIsentrFlow} when the pressure vanishes. They established that: any two-shock tends to a $\delta$-shock of \eqref{NoPressEulerEqu}; by contrast, any two-rarefaction tends to a two-contact-discontinuity with a vacuum state of \eqref{NoPressEulerEqu}. These results were extended for nonisentropic flows \cite{CHEN2004}, for the relativistic Euler equations for polytropic gases \cite{YIN2009}, and recently for the modified Chaplygin gas pressure law \cite{YANG2014}. Chen and Liu \cite{Vacuum} also mentioned that the behaviour of a $1$-shock combined with $2$-rarefaction  or a $1$-rarefaction combined with $2$-shock  can be deduced from the two above results. As far as we know, the proof and complete analysis of these two cases have not been clearly discussed in the literature on the degeneracy of the Euler equations. Since these two other cases must be dealt in a non-trivial manner, we will provide a complete proof.

The rest of this paper is organized as follows.  In section \ref{SectWavesCurves}, we recall the relations for shock and rarefaction waves for the Riemann problem of \eqref{EulerEqu}-\eqref{PressIsentrFlow}. In the sections \ref{ShockRar} and \ref{RarShock}, we analyse the behaviour a $1$-shock combined with $2$-rarefaction and a $1$-rarefaction wave combined with $2$-shock when the pressure vanishes, respectively. Numerical illustrations are carried out in section \ref{SectNumRes}.

\section{Rarefaction and shock curves}
\label{SectWavesCurves}
The Jacobian matrix of system \eqref{EulerEqu}-\eqref{PressIsentrFlow}
has two eigenvalues
\begin{equation}
\lambda_1=v-c, \quad \lambda_2=v+c,
\label{EulerEquEigVal}
\end{equation}
where $c=\sqrt{p'(\rho)}$ is called the \textit{sound speed}. The $j$-Riemann invariants are 
\begin{equation}
\phi_1=v+\dfrac{2}{\gamma-1}c \quad \text{and}\quad \phi_2=v-\dfrac{2}{\gamma-1}c.
\end{equation}
The Riemann problem for the isentropic Euler equations consists in the solutions of \eqref{EulerEqu}-\eqref{PressIsentrFlow} with initial data
\begin{equation}
(\rho,v)(x,0)=(\rho_{\pm},v_{\pm}),\quad \pm x>0,\\
\label{EulerEquInitCond}
\end{equation} 
where $\rho_-$, $\rho_+$ $\in\mathbb{R}^+$ and $v_-$, $v_+$  $\in\mathbb{R}$ are given constants. The solutions consist of rarefaction and shock waves \cite{Smoller}.  We briefly recall the relations for these waves for completeness.
\subsection{Rarefaction curves}
Given a state $(\rho_-,v_-)$, we look for the set of states $(\rho_{*},v_{*})$ that can be connected from the left to the state $(\rho_-,v_-)$ by a $1$-rarefaction wave. The states $(\rho_{*},v_{*})$ depend on the pressure $p$ which is parametrized by the coefficient $\kappa$. Therefore, in the following, these states are denoted by $(\rho_{*}^{\kappa},v_{*}^{\kappa})$. Using the fact that in a $j$-rarefaction wave, a $j$-Riemann invariant is constant (see Theorem 3.2 in \cite{Serre}) and $\lambda_1$ increases from the left to the right, one establishes that the state $(\rho_{*}^{\kappa},v_{*}^{\kappa})$ satisfies
\begin{equation}
v_{*}^{\kappa}=v_-+\dfrac{2}{\gamma-1}(c_--c_{*}^{\kappa})=v_-+\dfrac{2\sqrt{\kappa\gamma}}{\gamma-1}\big(\rho_-^{\frac{\gamma-1}{2}}-(\rho_{*}^{\kappa})^{\frac{\gamma-1}{2}}\big),\quad \rho_{*}^{\kappa}<\rho_-.
\label{1Rarefaction}
\end{equation} 
In a same way, one establishes that a state $(\rho_{*}^{\kappa},v_{*}^{\kappa})$ that can be connected from the right to a given state $(\rho_+,v_+)$ by a $2$-rarefaction wave satisfies 
\begin{equation}
v_{*}^{\kappa}=v_+-\dfrac{2\sqrt{\kappa\gamma}}{\gamma-1}\big(\rho_+^{\frac{\gamma-1}{2}}-(\rho_{*}^{\kappa})^{\frac{\gamma-1}{2}}\big),\quad \rho_{*}^{\kappa}<\rho_+.
\label{2Rarefaction}
\end{equation}  
\subsection{Shock curves}
Given a state $(\rho_-,v_-)$, we look for the set of states $(\rho_{*}^{\kappa},v_{*}^{\kappa})$ that can be connected from the left to the state $(\rho_-,v_-)$ by a $1$-shock wave. Combining the \textit{Rankine-Hugoniot conditions} \cite{Serre} and the \textit{Lax's entropy conditions} \cite{Smoller,Serre} for system \eqref{EulerEqu}-\eqref{PressIsentrFlow}, one establishes that the state  $(\rho_{*}^{\kappa},v_{*}^{\kappa})$  satisfies
\begin{equation}
v_{*}^{\kappa}=v_--\sqrt{\dfrac{\kappa\big((\rho_{*}^{\kappa})^{\gamma}-{\rho_-}^{\gamma}\big)}{\rho_{*}^{\kappa}\rho_-(\rho_{*}^{\kappa}-\rho_-)}}(\rho_{*}^{\kappa}-\rho_-), \quad \rho_{*}^{\kappa}>\rho_-.
\label{1Shock}
\end{equation}
In an analogous way, one establishes that a state $(\rho_{*}^{\kappa},v_{*}^{\kappa})$ that can be connected from the right to a given state $(\rho_+,v_+)$ by a $2$-shock satisfies
\begin{align}
v_{*}^{\kappa}=v_++\sqrt{\dfrac{\kappa\big({\rho_+}^{\gamma}-(\rho_{*}^{\kappa})^{\gamma}\big)}{\rho_+\rho_{*}^{\kappa}(\rho_+-\rho_{*}^{\kappa})}}(\rho_{*}^{\kappa}-\rho_+),\quad \rho_{*}^{\kappa}>\rho_+.
\label{2Shock}
\end{align}

\section{Behaviour of a $1$-shock and $2$-rarefaction Riemann solution as the pressure vanishes}
\label{ShockRar}
Let $v_->v_+$ and $\rho_{\pm}>0$. For $\kappa>0$, let $(\rho_*^{\kappa},\rho_*^{\kappa}v_*^{\kappa})$ be the intermediate state of a solution $(\rho^{\kappa},\rho^{\kappa}v^{\kappa})$ of the system \eqref{EulerEqu}-\eqref{PressIsentrFlow} with Riemann data \eqref{EulerEquInitCond}, in the sense that $v_-$ and $v_*^{\kappa}$ are connected by a $1$-shock wave, and $v_*^{\kappa}$ and $u_+$ are connected by a $2$-rarefaction wave. Then, this intermediate state is determined by \eqref{2Rarefaction} and \eqref{1Shock}, from which we immediately deduce that $\rho_+>\rho_-$. The following results also hold:
\begin{lemma}
\label{krefsrlem1}
\begin{align}
v_->v_*^{\kappa}> v_+,\text{ } \forall \kappa \in(0,\kappa_{sr}),\quad v_*^{\kappa}=v_+ \Longleftrightarrow \kappa=\kappa_{sr}:=\dfrac{\rho_-\rho_+(v_--v_+)^2}{(\rho_+^{\gamma}-\rho_-^{\gamma})(\rho_+-\rho_-)}.
\label{Ksr}
\end{align}
\end{lemma}
\begin{proof}
Let $\kappa>0$. We first prove the equivalence in \eqref{Ksr}. Assume that $v_*^{\kappa}=v_+$. From \eqref{2Rarefaction}, we get $\rho_*^{\kappa}=\rho_+$. Using the equalities $v_*^{\kappa}=v_+$ and $\rho_*^{\kappa}=\rho_+$ in \eqref{1Shock}, we obtain
\begin{equation}
v_+-v_-=-\sqrt{\dfrac{\kappa\big(\rho_+^{\gamma}-{\rho_-}^{\gamma}\big)}{\rho_+\rho_-(\rho_+-\rho_-)}}(\rho_+-\rho_-)=-\sqrt{\dfrac{\kappa\big(\rho_+^{\gamma}-{\rho_-}^{\gamma}\big)(\rho_+-\rho_-)}{\rho_+\rho_-}}
\end{equation}
which implies, by taking the square in both sides, that $\kappa=\kappa_{sr}$. 
Conversely, suppose that $\kappa=\kappa_{sr}$. Then 
\begin{equation}
v_--v_+=\sqrt{\dfrac{\kappa_{sr}\big(\rho_+^{\gamma}-{\rho_-}^{\gamma}\big)}{\rho_+\rho_-(\rho_+-\rho_-)}}(\rho_+-\rho_-).
\label{kref2}
\end{equation}
We claim that $\rho_*^{\kappa_{sr}}=\rho_+$. In fact, assume that  $\rho_*^{\kappa_{sr}}<\rho_+$.
By combining \eqref{2Rarefaction} and \eqref{1Shock}, we obtain
\begin{equation}
v_--v_+=\sqrt{\dfrac{\kappa_{sr}\big((\rho_{*}^{\kappa_{sr}})^{\gamma}-\rho_-^{\gamma}\big)}{\rho_*^{\kappa_{sr}}\rho_-(\rho_*^{\kappa_{sr}}-\rho_-)}}(\rho_*^{\kappa_{sr}}-\rho_-)-\dfrac{2\sqrt{\kappa_{sr}\gamma}}{\gamma-1}\big({\rho_+}^{\frac{\gamma-1}{2}}-(\rho_*^{\kappa_{sr}})^{\frac{\gamma-1}{2}}\big), \quad \rho_-<\rho_*^{\kappa_{sr}}<\rho_+.
\label{kref3}
\end{equation}
The monotonic increasing property of the function  $h_1:\rho\rightarrow(\rho^{\gamma}-\rho_-^{\gamma})(1-\frac{\rho_-}{\rho})$ implies that \eqref{kref2} and \eqref{kref3} cannot both be true.
Using the equality $\rho_*^{\kappa_{sr}}=\rho_+$ in \eqref{1Shock}, we obtain
\begin{equation}
v_--v_*^{\kappa_{sr}}=\sqrt{\dfrac{\kappa_{sr}\big(\rho_+^{\gamma}-{\rho_-}^{\gamma}\big)}{\rho_+\rho_-(\rho_+-\rho_-)}}(\rho_+-\rho_-),
\label{kref22}
\end{equation}
which, combined with \eqref{kref2}, implies that $v_*^{\kappa_{sr}}=v_*^{\kappa}=v_+$. Hence, the equivalence in \eqref{Ksr} holds. From \eqref{1Shock}, we get $v_*^{\kappa}<v_-$. It remains to prove that $v_*^{\kappa}>v_+$ for all $\kappa\in(0,\kappa_{sr})$. We proceed by contradiction. Suppose there exists $\kappa_1\in(0,\kappa_{sr})$ such that  $v_*^{\kappa_1}$ and $v_+$ are connected by a $2$-rarefaction wave. On the one hand, using the inequality $\kappa_1<\kappa_{sr}$ in the definition of $\kappa_{sr}$ in \eqref{Ksr}, one gets
\begin{equation}
v_--v_+>\sqrt{\dfrac{\kappa_1\big(\rho_+^{\gamma}-{\rho_-}^{\gamma}\big)}{\rho_+\rho_-(\rho_+-\rho_-)}}(\rho_+-\rho_-).
\label{kref4}
\end{equation}
On the other hand, as the intermediate state $(\rho_*^{\kappa_1},v_*^{\kappa_1})$ satisfies both \eqref{2Rarefaction} and \eqref{1Shock} then 
\begin{equation}
v_--v_+=\sqrt{\dfrac{\kappa_1\big((\rho_*^{\kappa_1})^{\gamma}-{\rho_-}^{\gamma}\big)}{\rho_*^{\kappa_1}\rho_-(\rho_*^{\kappa_1}-\rho_-)}}(\rho_*^{\kappa_1}-\rho_-)-\dfrac{2\sqrt{\kappa_1\gamma}}{\gamma-1}\big({\rho_+}^{\frac{\gamma-1}{2}}-(\rho_*^{\kappa_1})^{\frac{\gamma-1}{2}}\big),\quad \rho_-<\rho_*^{\kappa_1}<\rho_+.
\label{kref5}
\end{equation}
Again the monotonic increasing property of $h_1$ 
implies that \eqref{kref4} and \eqref{kref5} cannot both be true. Hence, $v_*^{\kappa}$ and $v_+$ are connected by a $2$-shock wave for all $\kappa\in(0,\kappa_{sr})$. This completes the proof.
\end{proof}

\begin{theorem}
\label{Theo1Shock2RarEulerEqu}
Let $v_->v_+$ and $\rho_{\pm}>0$. For some $\kappa>0$, assume that $(\rho^{\kappa},\rho^{\kappa} v^{\kappa})$ is a $1$-shock wave and $2$-rarefaction wave solution of \eqref{EulerEqu}-\eqref{PressIsentrFlow} with the Riemann data \eqref{EulerEquInitCond}. Then, when $\kappa\rightarrow0$, the solution $(\rho^{\kappa},\rho^{\kappa} v^{\kappa})$ tends to the $\delta$-shock solution of the pressureless gas system \eqref{NoPressEulerEqu} with the same Riemann data.
\end{theorem}
\begin{proof}
Lemma \eqref{krefsrlem1} says that the solution converts to a two-shock wave solution of system \eqref{EulerEqu}-\eqref{PressIsentrFlow} when $\kappa$ gets smaller than $\kappa_{sr}$. From Theorem 3.1 in \cite{Vacuum}, we conclude that, when $\kappa\rightarrow0$, the solution tends to the $\delta$-shock solution of system \eqref{NoPressEulerEqu} with the same initial data.
\end{proof}

\section{Behaviour of a $1$-rarefaction and $2$-shock Riemann solution as the pressure vanishes}
\label{RarShock}
Let $v_-<v_+$ and $\rho_{\pm}>0$. For $\kappa>0$, let $(\rho_*^{\kappa},\rho_*^{\kappa}v_*^{\kappa})$ be the intermediate state of a solution $(\rho^{\kappa},\rho^{\kappa}v^{\kappa})$ of the system \eqref{EulerEqu}-\eqref{PressIsentrFlow} with Riemann data \eqref{EulerEquInitCond}, in the sense that $v_-$ and $v_*^{\kappa}$ are connected by a $1$-rarefaction wave, and $v_*^{\kappa}$ and $u_+$ are connected by a $2$-shock wave. Then, this intermediate state is determined by \eqref{1Rarefaction} and \eqref{2Shock} which  imply that  $\rho_->\rho_+$. The following results also hold:
\begin{lemma}
\label{krefrslem1}
\begin{equation}
v_-< v_*^{\kappa}<v_+,\text{ } \forall \kappa\in(0,\kappa_{rs}),\quad v_*^{\kappa}=v_+ \Longleftrightarrow \kappa=\kappa_{rs}:=\left(\dfrac{(v_+-v_-)(\gamma-1)}{2\sqrt{\gamma}(\rho_-^{\frac{\gamma-1}{2}}-\rho_+^{\frac{\gamma-1}{2}})}\right)^2.
\label{Krs}
\end{equation}
\end{lemma}
\begin{proof}
Let $\kappa>0$. We first prove the equivalence in \eqref{Krs}. Assume that $v_*^{\kappa}=v_+$. Using this equality in \eqref{2Shock}, we get $\rho_*^{\kappa}=\rho_+$. Now, taking $v_*^{\kappa}=v_+$ and $\rho_*^{\kappa}=\rho_+$ in \eqref{1Rarefaction}, we obtain
\begin{equation}
v_+=v_-+\dfrac{2\sqrt{\kappa\gamma}}{\gamma-1}(\rho_-^{\frac{\gamma-1}{2}}-\rho_+^{\frac{\gamma-1}{2}})
\end{equation}
which implies  that $\kappa=\kappa_{rs}$. Conversely, suppose that $\kappa=\kappa_{rs}$. Then 
\begin{equation}
v_+-v_-=\dfrac{2\sqrt{\kappa_{rs}\gamma}}{\gamma-1}(\rho_-^{\frac{\gamma-1}{2}}-\rho_+^{\frac{\gamma-1}{2}}).
\label{kref8}
\end{equation}
We claim that $\rho_*^{\kappa_{rs}}=\rho_+$. In fact, assume that $\rho_+<\rho_*^{\kappa_{rs}}$.
Combining \eqref{1Rarefaction} and \eqref{2Shock}, we obtain
\begin{equation}
v_+-v_-=\dfrac{2\sqrt{\kappa_{rs}\gamma}}{\gamma-1}\big({\rho_-}^{\frac{\gamma-1}{2}}-(\rho_*^{\kappa_{rs}})^{\frac{\gamma-1}{2}}\big)-\sqrt{\dfrac{\kappa_{rs}\big((\rho_{*}^{\kappa_{rs}})^{\gamma}-\rho_+^{\gamma}\big)}{\rho_*^{\kappa_{rs}}\rho_+(\rho_*^{\kappa_{rs}}-\rho_+)}}(\rho_*^{\kappa_{rs}}-\rho_+), \quad \rho_+<\rho_*^{\kappa_{rs}}<\rho_-.
\label{kref9}
\end{equation}
The monotonic increasing property of the function $h_2:\rho\rightarrow\rho^{\frac{\gamma-1}{2}}-\rho_-^{\frac{\gamma-1}{2}}$
implies that  \eqref{kref8} and \eqref{kref9} cannot both  be true.
Taking the equality $\rho_*^{\kappa_{rs}}=\rho_+$  in \eqref{1Rarefaction}, we obtain 
\begin{equation}
v_*^{\kappa_{rs}}-v_-=\dfrac{2\sqrt{\kappa_{rs}\gamma}}{\gamma-1}(\rho_-^{\frac{\gamma-1}{2}}-\rho_+^{\frac{\gamma-1}{2}}),
\label{kref18}
\end{equation}
which, combined with \eqref{kref8}, implies that $v_*^{\kappa}=v_*^{\kappa_{sr}}=v_+$. Hence, the equivalence in \eqref{Krs} holds. From \eqref{1Rarefaction}, we get $v_*^{\kappa}>v_-$. It remains to prove that $v_*^{\kappa} < v_+$ for all  $\kappa\in(0,\kappa_{rs})$. We proceed by contradiction. Suppose that there exists $\kappa_2\in(0,\kappa_{rs})$ such that $v_*^{\kappa_2}$ and $v_+$ are connected by a $2$-shock wave. Then, using the inequality $\kappa_2<\kappa_{rs}$ in the definition of $\kappa_{rs}$ in \eqref{Krs}, we get 
\begin{equation}
v_+-v_->\dfrac{2\sqrt{\kappa_2\gamma}}{\gamma-1}(\rho_-^{\frac{\gamma-1}{2}}-\rho_+^{\frac{\gamma-1}{2}}).
\label{kref10}
\end{equation}
As the intermediate state $(\rho_*^{\kappa_2},v_*^{\kappa_2})$ satisfies both \eqref{1Rarefaction} and \eqref{2Shock} then  
\begin{equation}
v_+-v_-=\dfrac{2\sqrt{\kappa_2\gamma}}{\gamma-1}\big({\rho_-}^{\frac{\gamma-1}{2}}-(\rho_*^{\kappa_2})^{\frac{\gamma-1}{2}}\big)-\sqrt{\dfrac{\kappa_2\big((\rho_*^{\kappa_2})^{\gamma}-{\rho_+}^{\gamma}\big)}{\rho_*^{\kappa_2}\rho_+(\rho_*^{\kappa_2}-\rho_+)}}(\rho_*^{\kappa_2}-\rho_+), \quad \rho_+<\rho_*^{\kappa_2}<\rho_-.
\label{kref11}
\end{equation}
Again, the monotonic increasing property of $h_2$ implies that \eqref{kref10} and \eqref{kref11} cannot both be true. Hence,  $v_*^{\kappa}$ and $v_+$ are connected by a $2$-rarefaction wave for all $\kappa\in(0,\kappa_{rs})$. This completes the proof.
\end{proof}

\begin{theorem}
\label{Theo1Rar2ShockEulerEqu}
Let $v_-<v_+$ and $\rho_{\pm}>0$. For some $\kappa>0$, assume that $(\rho^{\kappa},\rho^{\kappa} v^{\kappa})$ is a $1$-rarefaction wave and $2$-shock wave  solution of \eqref{EulerEqu}-\eqref{PressIsentrFlow} with the Riemann initial data \eqref{EulerEquInitCond}. Then, when $\kappa\rightarrow0$, the solution $(\rho^{\kappa},\rho^{\kappa} v^{\kappa})$ tends to the two-contact-discontinuity solution  of the pressureless gas system \eqref{NoPressEulerEqu} with the same Riemann data.
\end{theorem}
\begin{proof}
Lemma \eqref{krefrslem1} says that the solution converts to a two-rarefaction wave solution of system \eqref{EulerEqu}-\eqref{PressIsentrFlow} when $\kappa$ gets smaller than $\kappa_{rs}$. From the result established by Chen and Liu (see Section 4 in \cite{Vacuum}), we conclude that, when $\kappa\rightarrow0$, the solution tends to the two-contact-discontinuity solution of system \eqref{NoPressEulerEqu} with the same initial data.
\end{proof}

\section{Numerical illustrations}
\label{SectNumRes}
This section is devoted to the numerical illustration of the theoretical analysis.  We use the modified Lax Friedrich scheme from \cite{Tadmor} to discretize the equations \eqref{EulerEqu}-\eqref{PressIsentrFlow} and we take $\gamma=1.4$.

We first illustrate the behaviour of a solution composed of a $1$-rarefaction wave and a $2$-shock wave when the pressure coefficient $\kappa$ vanishes. The Riemann data are 
\begin{equation}
(\rho,v)(x,0)=
\left\lbrace
\begin{aligned}
&(1.0,0.8),\quad \text{for}\quad x<0,\\
&(0.5,1.0),\quad \text{for}\quad x>0,\\
\end{aligned}
\right.
\end{equation}
from which, we calculate by using \eqref{Krs}, the coefficient $\kappa_{rs}\approx 0.07$. Numerical results, when the pressure coefficient is decreased, are represented in Figure \ref{FigVacuumFormationInEulerEqu}. When $\kappa=\kappa_{rs}$, we observe that the $2$-shock wave disappears. When $\kappa$ gets smaller than $\kappa_{rs}$, the solution converts to a two-rarefaction wave. When $\kappa$ tends to zero, the two-rarefaction wave tends to a two-contact-discontinuity, whose intermediate state between the two contact discontinuities tends to a vacuum state.

\begin{figure}[!h]
\centering
\begin{tabular}{ccc}
\includegraphics[scale=0.35]{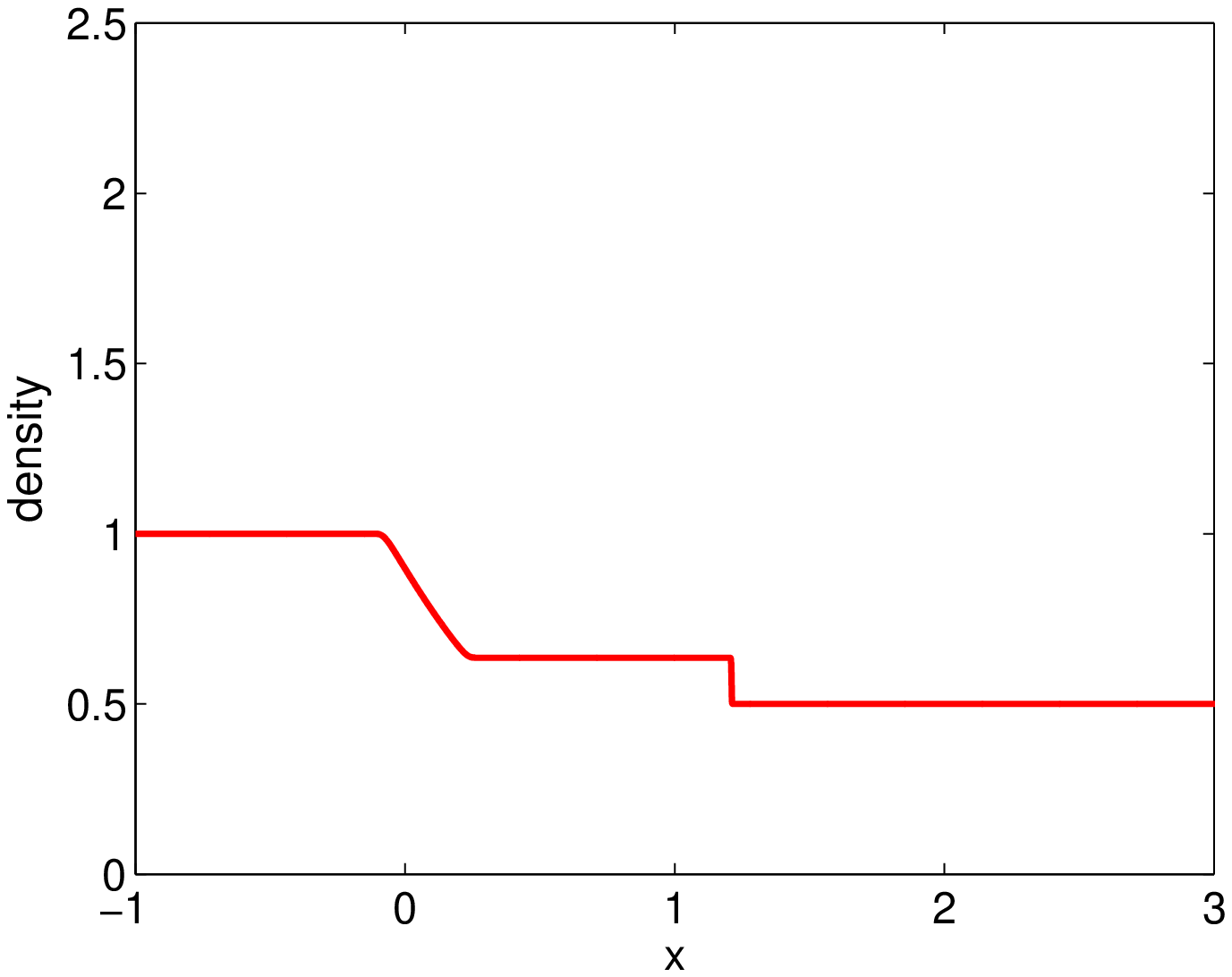}&\includegraphics[scale=0.35]{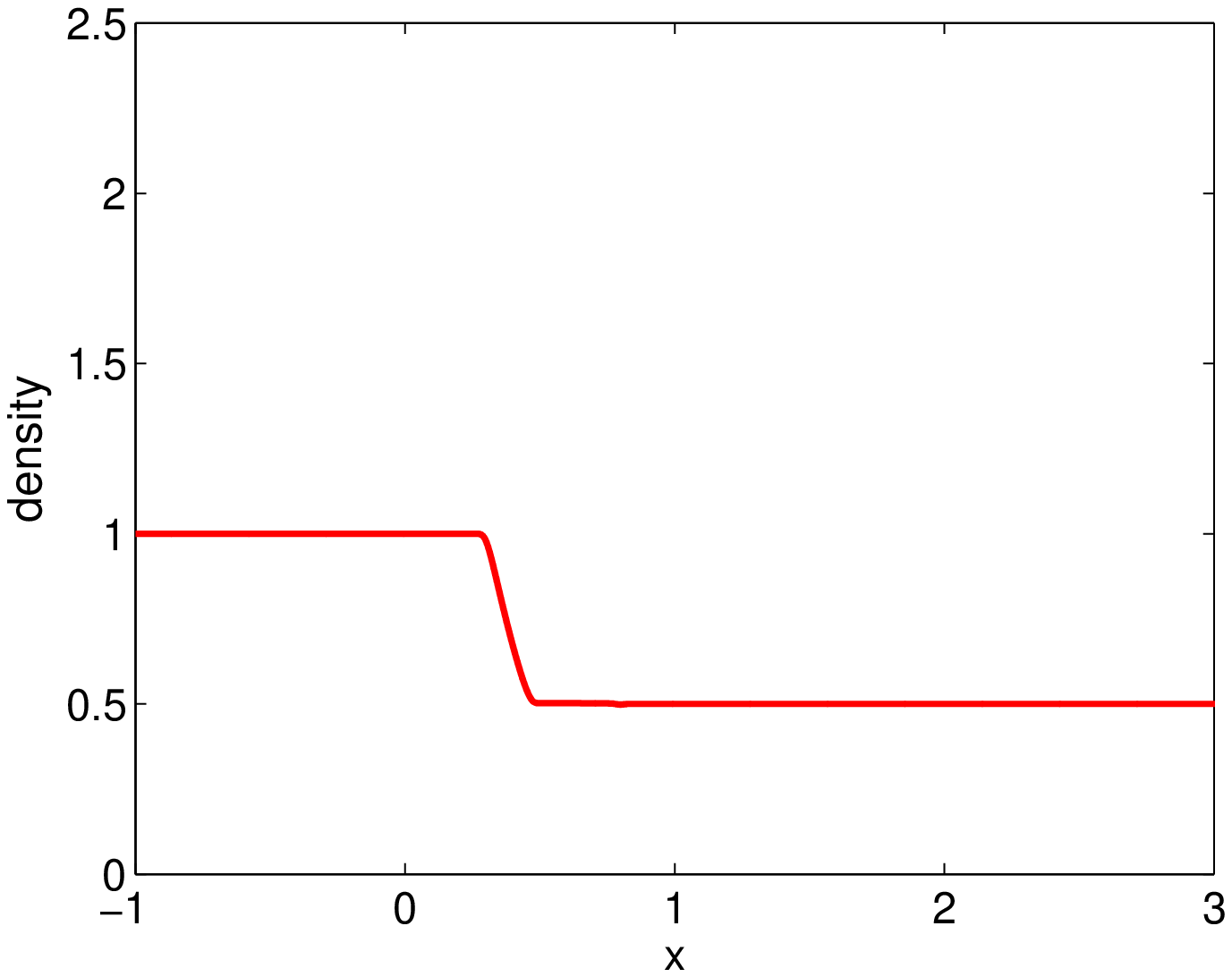}&\includegraphics[scale=0.35]{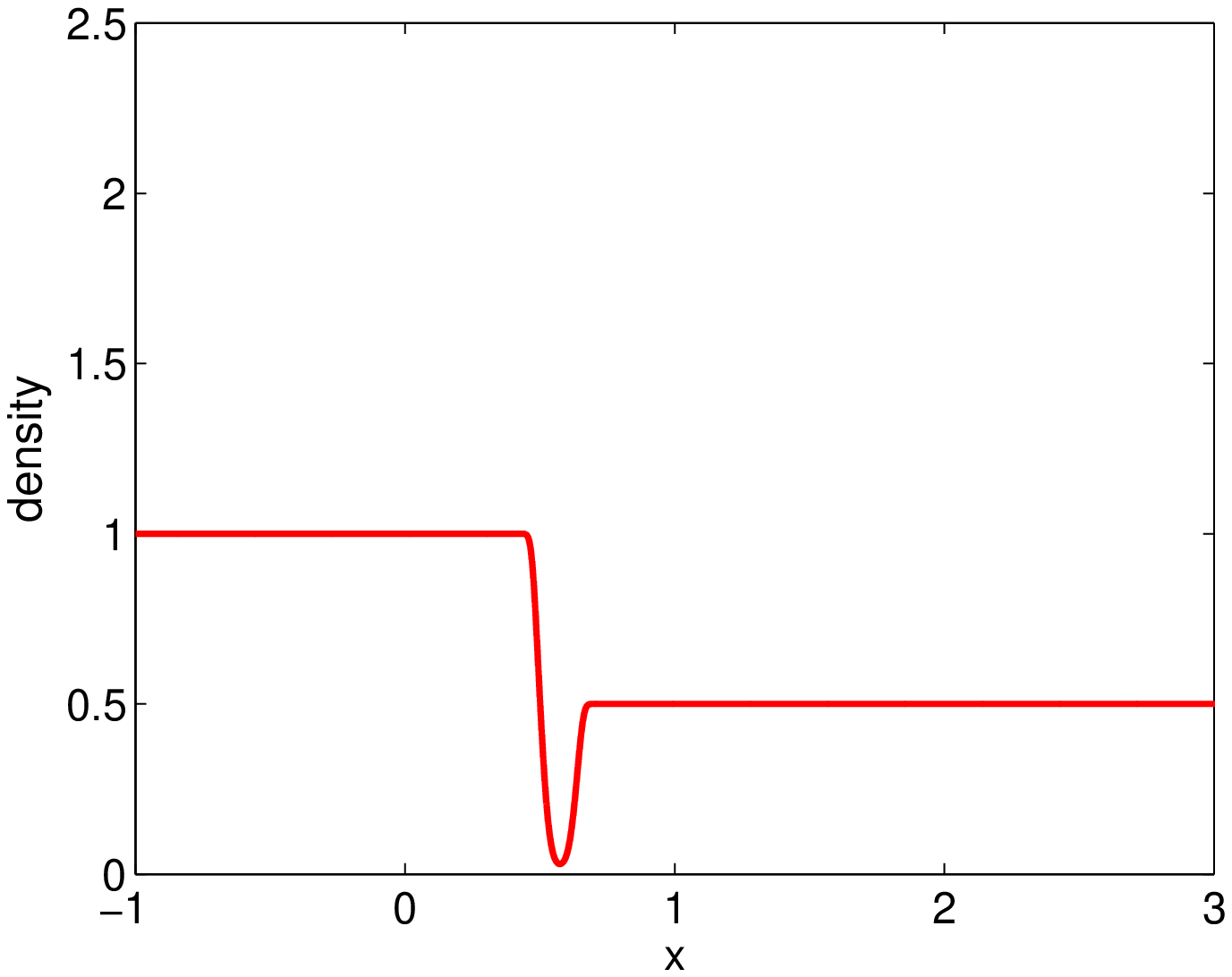}\\
\includegraphics[scale=0.35]{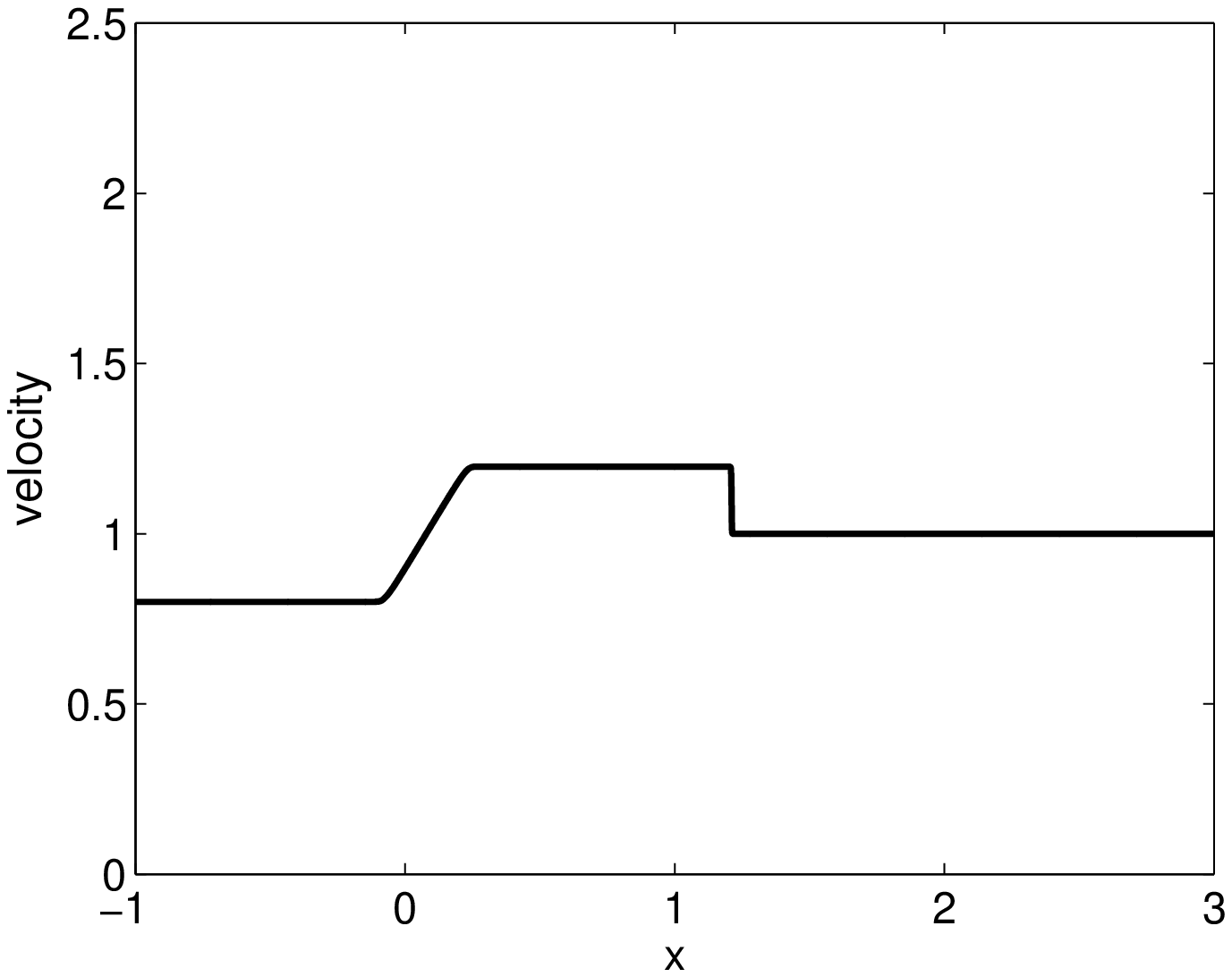}&\includegraphics[scale=0.35]{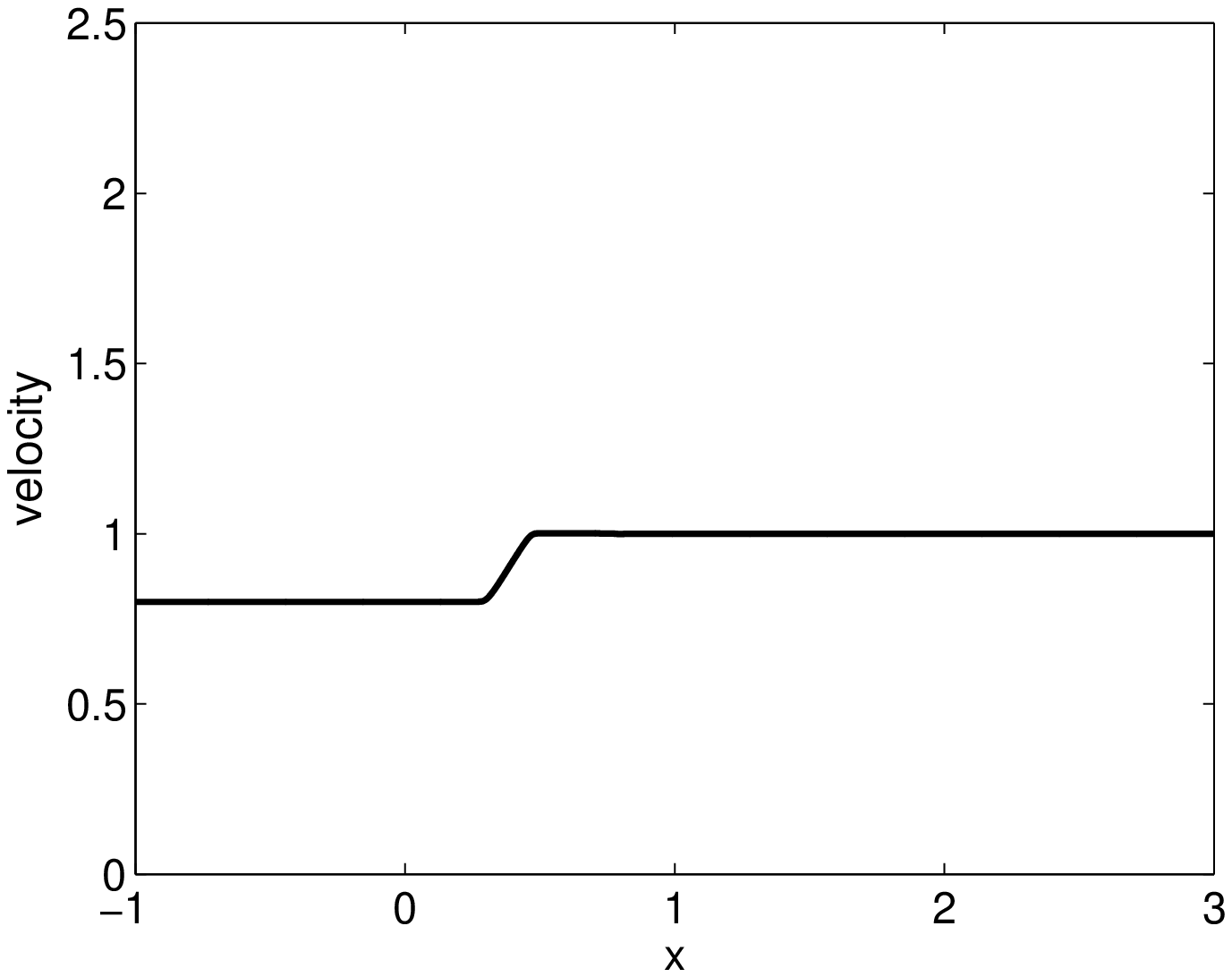}&\includegraphics[scale=0.35]{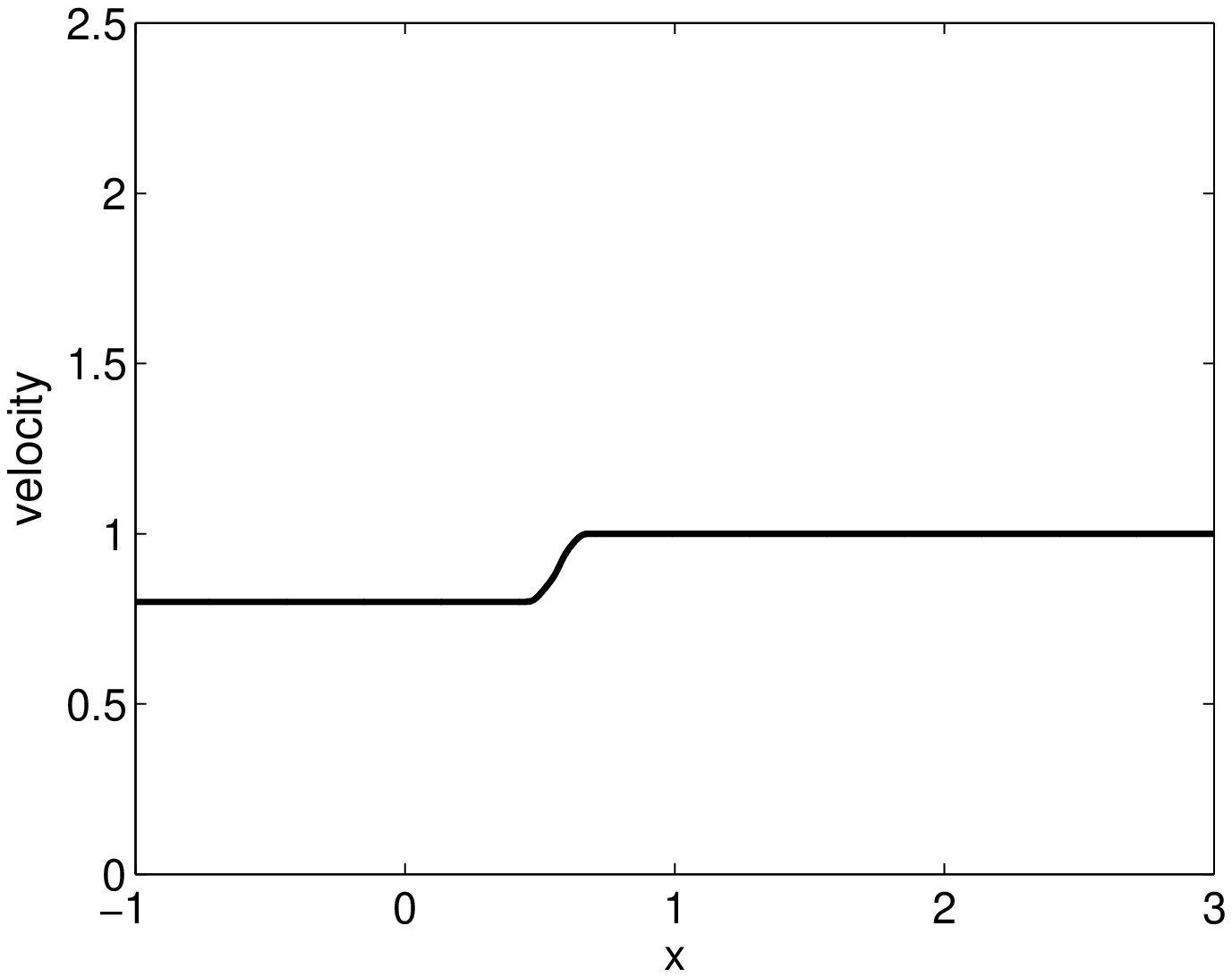}\\
$\kappa=0.6$ & $\kappa=\kappa_{rs}\approx 0.068$& $\kappa=0.001$
\end{tabular}
\caption{Formation of a vacuum state in the isentropic Euler equations as the pressure vanishes. $t=0.63$, $\gamma=1.4$, $\Delta x =10^{-4}$ and $\Delta t=2\times10^{-5}$.}
\label{FigVacuumFormationInEulerEqu}
\end{figure} 

The behaviour of a solution composed of a $1$-shock wave and a $2$-rarefaction wave is illustrated using the Riemann data
\begin{equation}
(\rho,v)(x,0)=
\left\lbrace
\begin{aligned}
&(0.2,1.5),\quad \text{for}\quad x<0,\\
&(0.7,1.0),\quad \text{for}\quad x>0,\\
\end{aligned}
\right.
\end{equation}
from which, we calculate by using \eqref{Ksr}, the coefficient $\kappa_{sr}\approx 0.14$. Numerical solutions, when the pressure coefficient $\kappa$ is decreased, are shown in Figure \ref{FigDeltaShockFormationInEulerEqu}. When $\kappa=\kappa_{sr}$, we notice that the $2$-rarefaction wave disappears. When $\kappa$ gets smaller than $\kappa_{sr}$, the solution converts to a two-shock wave that tends to a delta-shock wave as the pressure coefficient $\kappa$ progressively vanishes.

\begin{figure}[!h]
\centering
\begin{tabular}{ccc}
\includegraphics[scale=0.35]{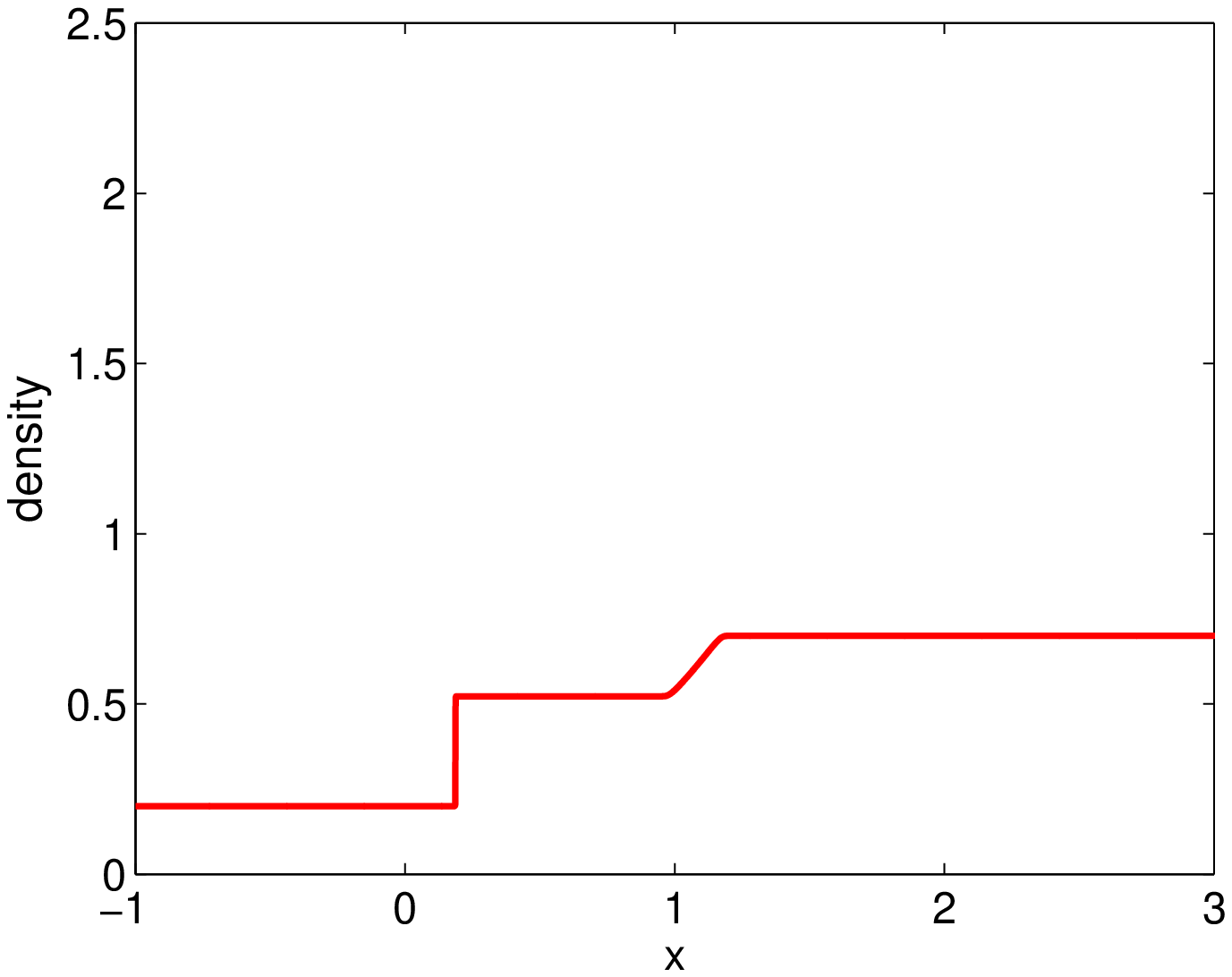}&\includegraphics[scale=0.35]{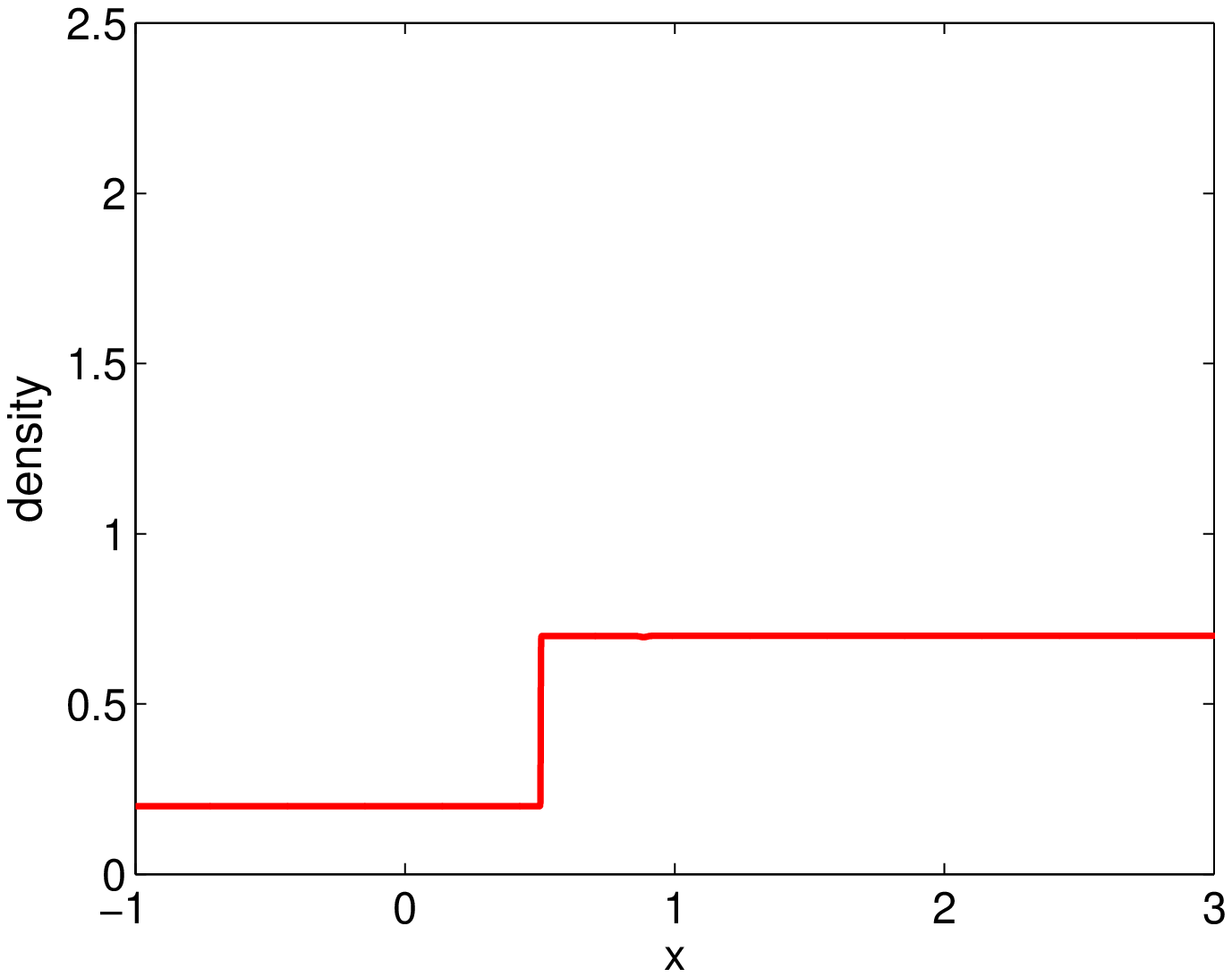}&\includegraphics[scale=0.35]{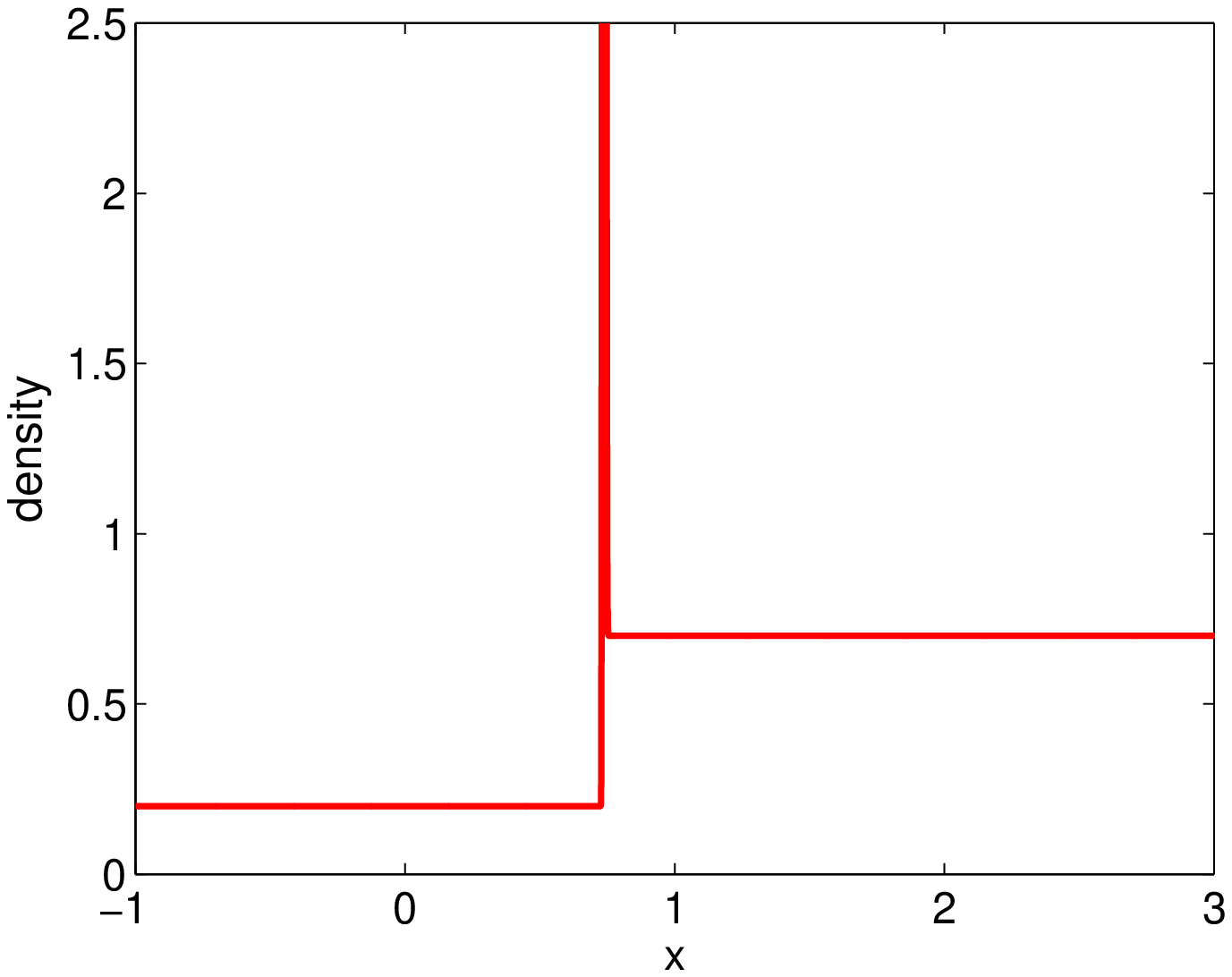}\\
\includegraphics[scale=0.35]{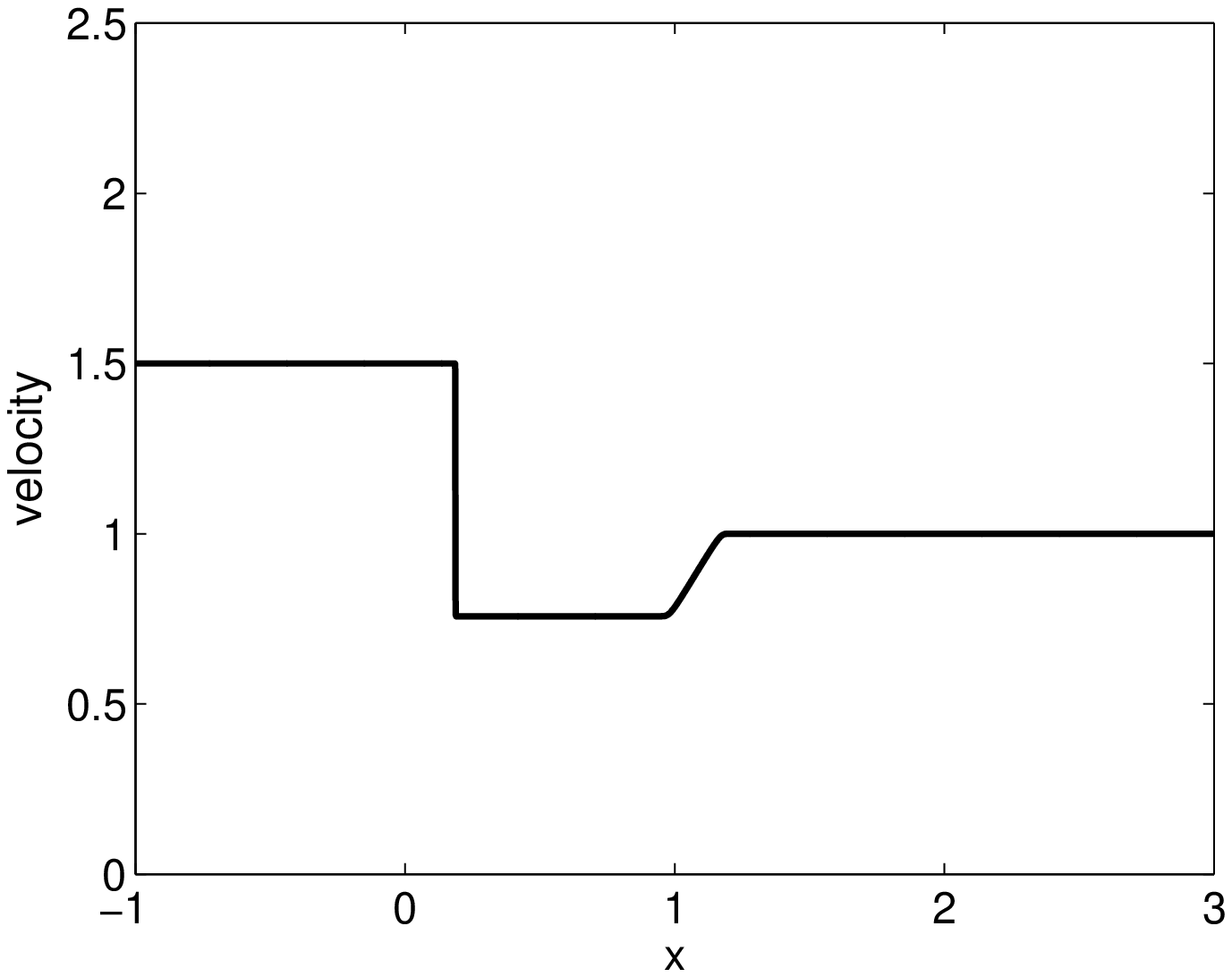}&\includegraphics[scale=0.35]{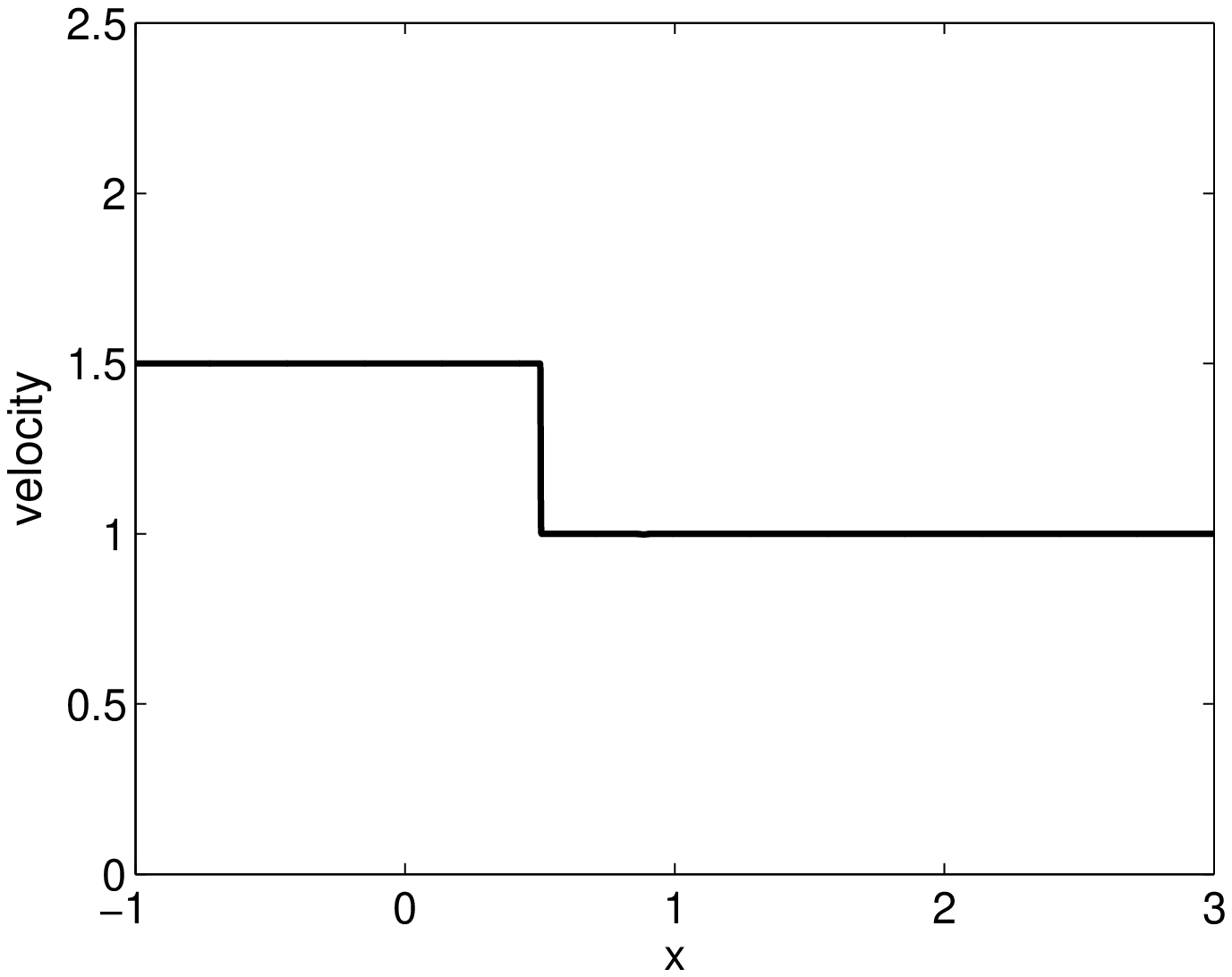}&\includegraphics[scale=0.35]{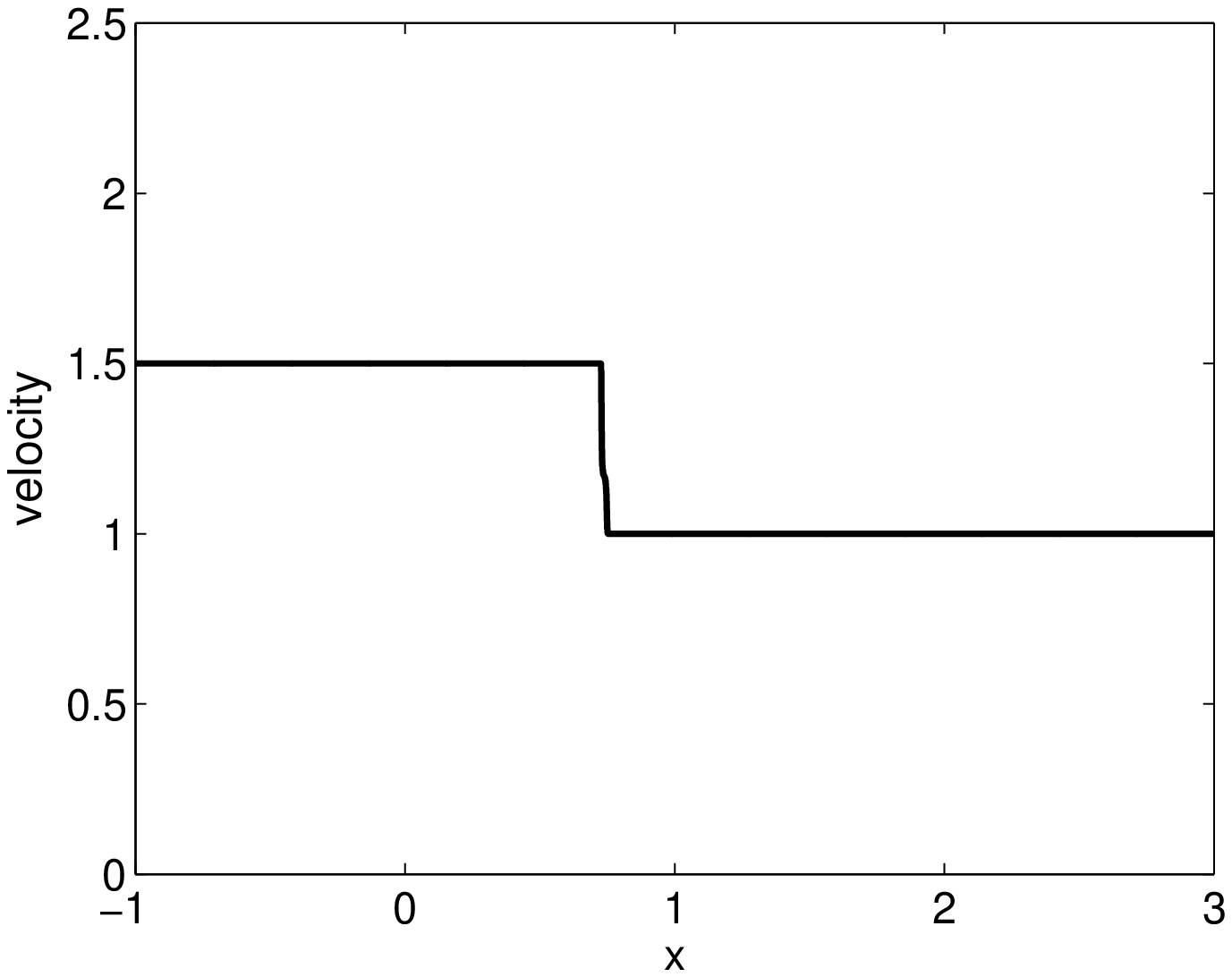}\\
$\kappa=0.6$ & $\kappa=\kappa_{sr}\approx0.14$&$\kappa=0.001$\\\\
\end{tabular}
\caption{Formation of a delta-shock wave in the isentropic Euler equations as the pressure vanishes. $t=0.63$,  $\gamma=1.4$, $\Delta x =10^{-4}$ and $\Delta t=2\times10^{-5}$.}
\label{FigDeltaShockFormationInEulerEqu}
\end{figure}

\newpage


\bibliographystyle{plain} 
\bibliography{biblio}

\begin{thebibliography}{10}

\bibitem{Bouchut2}
F.~Bouchut.
\newblock On zero pressure gas dynamics.
\newblock In {\em Advances in kinetic theory and computing, Vol.\ 22}, pages
  171--190. 1994.

\bibitem{Vacuum}
G.-Q. Chen and H.~Liu.
\newblock {Formation of delta-shocks and vacuum states in the vanishing
  pressure limit of solutions to the {E}uler equations for isentropic fluids}.
\newblock {\em SIAM J. Math. Anal.}, 34:925--938, 2003.

\bibitem{CHEN2004}
G.-Q. Chen and H.~Liu.
\newblock Concentration and cavitation in the vanishing pressure limit of
  solutions to the {E}uler equations for nonisentropic fluids.
\newblock {\em Phys. D}, 189:141--165, 2004.

\bibitem{weinan1996}
W.~E, Yu.G. Rykov, and Ya.G. Sinai.
\newblock Generalized variational principles, global weak solutions and
  behavior with random initial data for systems of conservation laws arising in
  adhesion particle dynamics.
\newblock {\em Comm. Math. Phys.}, 177:349--380, 1996.

\bibitem{Serre}
D.~Serre.
\newblock {\em Systems of {C}onservation {L}aws: {H}yperbolicity, {E}ntropies,
  {S}hock {W}aves}.
\newblock Cambridge University Press, Cambridge, 1999.

\bibitem{Sheng}
W.~Sheng and T.~Zhang.
\newblock {\em The {R}iemann {P}roblem for the {T}ransportation {E}quations in
  {G}as {D}ynamics}.
\newblock American Mathematical Society, 1999.

\bibitem{Smoller}
J.~Smoller.
\newblock {\em Shock {Wa}ves and {R}eaction-{D}iffusion {E}quations}.
\newblock Springer-{V}erlag, {N}ew {Y}ork, 1994.

\bibitem{Tadmor}
E.~Tadmor.
\newblock Numerical viscosity and the entropy condition for conservative
  difference schemes.
\newblock {\em Math. Comp.}, 43:369--381, 1984.

\bibitem{YANG2014}
H.~Yang and J.~Wang.
\newblock Delta-shocks and vacuum states in the vanishing pressure limit of
  solutions to the isentropic {E}uler equations for modified {C}haplygin gas.
\newblock {\em J. Math. Anal. Appl.}, 413:800--820, 2014.

\bibitem{YIN2009}
G.~Yin and W.~Sheng.
\newblock Delta shocks and vacuum states in vanishing pressure limits of
  solutions to the relativistic {E}uler equations for polytropic gases.
\newblock {\em J. Math. Anal. Appl.}, 355:594--605, 2009.

\end{thebibliography}

\end{document}